\documentclass[12pt]{amsart}
\usepackage{amsmath}
\usepackage{amssymb,amsthm}
\usepackage{graphics}

\newtheorem{teo}{Theorem}[section]
\newtheorem{proposition}[teo]{Proposition}

\newtheorem{lemma}[teo]{Lemma}
\newtheorem{definition}[teo]{Definition}
\newtheorem{remark}[teo]{Remark}
\newtheorem{example}[teo]{Examples}

\numberwithin{equation}{section}

\begin{document}
\title[Fuglede-Kadison determinants]{Fuglede-Kadison determinants for operators in the von Neumann algebra of an equivalence relation }
%\urladdr{}
\author[Georgescu]{Catalin Georgescu}
\address{Department of Mathematical Sciences, The University of South Dakota, Dakota Hall, 414 East Clark Street, Vermillion, SD 57069}
\email{catalin.georgescu@usd.edu}
%\urladdr{}
\author[Picioroaga]{Gabriel Picioroaga}
\address{Department of Mathematical Sciences, The University of South Dakota, Dakota Hall, 414 East Clark Street, Vermillion, SD 57069}
\email{gabriel.picioroaga@usd.edu}
%\urladdr{}
\thanks{This research was partially supported by The Office of Research and the College of Arts and Sciences of the University of South Dakota under a 2011 Research Excellence Grant.}
\keywords{equivalence relation, von Neumann algebra, Fuglede-Kadison determinant}
\subjclass[2000]{47C15,47A35,47B47}

%\keywords{}
%\subjclass[2000]{}
\date{}

\begin{abstract}
We calculate the Fuglede-Kadison determinant for operators of the form $\sum_{i=1}^n M_{f_i}L_{g_i}$ where $L_{g_i}$ are unitaries or partial isometries coming from Borel (partial) isomorphisms $g_i$ on a probability space which generate an ergodic equivalence relation, and $M_{f_i}$ are multiplication operators. We obtain formulas for the cases when the relation is treeable or the $f_i$'s and $g_i$'s satisfy some restrictions.
\end{abstract}

\maketitle

\section{Introduction.}
\par The determinant of an (invertible) operator was first introduced in 1952 in the paper \cite{FK}. The notion generalizes the usual determinant and can be considered for any operator in a finite von Neumann algebra $(M,\tau)$ with a faithful normal trace. To summarize the construction let
$T\in M$ be normal and $|T|=\sqrt{T^*T}$. By the spectral theorem one can represent $T$ as an integral where $E(\lambda)$ is a projection-valued measure:
$$T=\int_{\sigma(T)}\lambda dE(\lambda)$$
In this setting $\mu_T=\tau\circ E$ becomes a probability measure on the complex plane whose support is the spectrum $\sigma(T)$.
\par For any $T\in M$ the Fuglede-Kadison determinant is defined by :
$$\Delta(T)=\exp\left(\int_{0}^{\infty}\log t \,d\mu_{|T|}\right)$$
Notice that if $T$ is invertible then
$$\Delta(T)=\exp(\tau(\log|T|))$$
where $\log|T|$ comes from functional calculus.
\par The determinant has been used recently in the calculation of Brown measures and applied to the invariant subspace problem by Haagerup and Schultz (see \cite{HSc} where $\Delta$ is defined for a general class of operators), and it has been connected to the entropy of certain algebraic actions in the work of Bowen, \cite{ Bow}, Deninger \cite{Den2} and Kerr and Li \cite{KLi}. We point out the paper \cite{Den1} where the determinant is calculated for certain operators in the von Neumann algebra generated by the action of the group $\mathbf{Z}$ on a probability space. In this paper we will extend the setting in which such computations can be carried out. Thus,  instead of a $\mathbf{Z}$ action, we will consider a standard preserving equivalence relation and the von Neumann algebra it generates by the Feldman-Moore construction (\cite{FMII}). In order to make sure we deal with a finite von Neumann algebra we will require the equivalence relation be ergodic which together with (finite) measure preservation will bring in more, namely a $II_1$ factor. We set out to calculate the determinant for  operators of the form $\sum_{i=1}^N M_{f_i}L_{g_i}$ where the $g_i$ are Borel partial isomorphisms that (partly) generate the equivalence relation and the $f_i$ are bounded. A more detailed set-up follows.
\par Let $(X,\mathcal{B}, \mu)$ be a Borel standard probability space without atoms, $\{A\}_{i \in I}$ and $\{B\}_{i \in I}$ two families of measurable subsets of $X$, and $\Lambda=\{g_i:A_i \to B_i\mbox{ }|\mbox{ }i\in I\}$ a family of measure preserving bijections. Assuming further that the index set $I$ is at most countable let $\mathcal{R}_{\Lambda}$ be the equivalence relation generated by the $g_i$  i.e. $(x,y)\in\mathcal{R}_{\Lambda}$ if and only if $x=y$ or there exists a map $\omega=g_{i_1}^{\epsilon_1}g_{i_2}^{\epsilon_2}\dots g_{i_k}^{\epsilon_k}$  such that the domain of $\omega$ contains $x$ and $\omega x=y$, where all exponents $\epsilon_{t}=\pm 1$. For simplicity we will denote the fact that two elements $x$ and $y$ are equivalent by $x\sim y$. An equivalence relation $\mathcal{R}$ is (SP1) if it  satisfies the properties listed below:
\vspace{0.09in}
\par (S) Almost each orbit $\mathcal{R}[x]=\{y\in X\mbox{, }y\sim x\}$ is at most countable and $\mathcal{R}$ is a Borel set of $X\times X$.
\par (P) For any $T\in Aut(X,\mu)$ such that $\mbox{graph}\,T\subset\mathcal{R}$ we have that $T$ preserves the measure $\mu$.
\vspace{0.09in}\\
It can be shown that an equivalence relation $\mathcal{R}_{\Lambda}$ with $\Lambda$ as above is (SP1). This is needed in Proposition \ref{trace} and Theorem \ref{r1}. \\ 
 If $A$ is measurable we will denote its \textsf{saturation} with respect to $\mathcal{R}$ by $\mathcal{R}[A]=\{y\in X, y\sim x\mbox{ for some }x\in A\}$.
\begin{definition}
The equivalence relation $\mathcal{R}$ is called \textsf{ergodic} if for any measurable set $A \in \mathcal{B}$ with $A=\mathcal{R}(A)$ we have $\mu(A)=0$ or $\mu(A)=1$.
\end{definition}
\begin{definition}
A family of measure preserving bijections  $\Lambda=\{g_i:A_i \to B_i\mbox{ }|\mbox{ }i\in I\}$ is a \textsf{treeing} if 
$\mu\{x\in Dom(\omega) \mid \omega(x)=x \}=0$ for every non trivial reduced word $\omega$  (i.e. $\omega=g_1^{\varepsilon_1}g_2^{\varepsilon_2}\ldots g_k^{\varepsilon_k}$ with $\varepsilon_i=\pm 1$ and if $g_i=g_{i+1}$ then $\varepsilon_i=\varepsilon_{i+1}$) with domain $Dom(\omega)$. 
An equivalence relation $\mathcal{R}$ is called \textsf{treeable} if there exists a treeing $\Lambda$ such that $\mathcal{R}=\mathcal{R}_{\Lambda}$.  
\end{definition}
The Feldman-Moore construction is based on the following Hilbert space (see \cite{FMII} ) :
\[
L^2(\mathcal{R})=\{\Psi : \mathcal{R} \to \mathbb{C} \quad \text{such that}\quad ||\Psi||_2<\infty \}
\]
where the norm $||\cdot||_2$ comes from the scalar product:
\[
<\Psi,\Phi>= \int_X \sum_{z\sim x} \Psi(x,z) \overline{\Phi(x,z)} d\mu(x)
\]
For $\omega$ a reduced word and $f \in L^{\infty}(X)$ we consider the operators\\ $L_{\omega},M_f:L^2(\mathcal{R}) \to L^2(\mathcal{R})$
\begin{equation}
(L_{\omega}\Psi)(x,y)=\chi_{D_{\omega}}(x)\Psi(\omega^{-1}x,y),\quad D_{\omega}\mbox{ the domain of }\omega^{-1}
\end{equation}
\begin{equation}
(M_f\Psi)(x,y)=f(x)\Psi(x,y)
\end{equation}
If $\omega$ is defined on almost all of $X$ then $L_{\omega}$ is unitary.
Direct computations show the following basic properties:
\begin{equation}
L_{\omega}^*=L_{\omega^{-1}} \quad \text{and} \quad L_{{\omega}^{-1}} \circ L_{\omega}=M_{\chi_{D_{\omega}}}
\end{equation}
Also if $\omega=g_1g_2\ldots g_k$ then $L_{\omega}=L_{g_1}L_{g_2}\ldots L_{g_k}$.
These imply that:
\begin{equation}\label{normLg}
||L_{\omega}|| \leqslant 1
\end{equation}
Clearly $M_f$ is a multiplication operator and
\begin{equation}\label{normMf}
||M_f|| = ||f||_{\infty}
\end{equation}
The closure in the weak topology of the linear span of all products (compositions) made with the operators $L_{g}$ and $M_f$ is a von  Neumann algebra, denoted by $\mathcal{M }(\mathcal{R})$. When $\mathcal{R}$ is ergodic and (SP1) this algebra becomes a $II_1$-factor. Its trace is given by :
\[
\tau(T)=<T\delta_0,\delta_0>
\]
where $\delta_0$ is the characteristic function of the diagonal of $\mathcal{R}$.

\begin{example} If $\Gamma$ is a countable group then any free, ergodic action on a standard probability space give rise to an (SP1) equivalence relation on that space. Such a relation is generated for example by $g:X\rightarrow X$, $\forall g\in\Gamma$, however it may be done in a more efficient way in case $\Gamma $ possesses generators. In turn, there is always an action around: any countable group acts freely and ergodically on $X=\{0,1\}^{\Gamma}$ equipped with the product measure by means of the Bernoulli shifts. Many countable groups gives rise to treeable, ergodic equivalence relations such as those coming from actions of the free groups or the amenable ones. Moreover it is not necessary that the domains of the generators be all of $X$: there are groups of non-integer cost hence some of their generators must be defined on ''chopped off'' measurable pieces. For a detailed discussion on costs of groups and treeability, and more examples of (SP1) equivalence relations we refer the reader to \cite{Gab}. Finally we note that if the equivalence relation comes from the action of a countable group $\alpha:\Gamma\rightarrow \mbox{Aut}(X,\mu)$ then the Feldman-Moore construction yields the usual
cross-product von Neumann algebra $L^{\infty}(X)\rtimes_{\alpha}\Gamma$.
\end{example}
Going back to the determinant let us recall the following familiar properties:  \\
$$\mbox{For }  n\times n \mbox{ matrices } \tau=\mbox{Tr}/n: \, \Delta(T)=\sqrt[n]{|\det T|}$$
$$\Delta(ST)=\Delta(S)\Delta(T)$$
$$\Delta(S)=\Delta(|S|)=\Delta(S^*)$$
$$\Delta(U)=1\mbox{ where }U\mbox{ is unitary}$$
$$\Delta(\lambda I)=|\lambda|$$
Although the determinant is not necessarily continuous it is upper-semicontinuous both in the strong operator topology and the norm topology.
\par The following formula can be proven using the description of the spectral resolution of a multiplication operator and functional calculus (see the proof of theorem \ref{r2} below):
\begin{equation}\label{det of Mf}
\log \Delta (M_f) = \int_X \log|f|\, d\mu(x)
\end{equation}

The goal of the present work is to compute the Fuglede-Kadison determinant of the operator $T\in\mathcal{M(R)}$:

\begin{equation}\label{operator}
T=\sum_{i=1}^N M_{f_i}L_{g_i}.
\end{equation}
We will accomplish this under some restrictions. First we will recall the following proposition due to Deninger (see \cite{Den1}):

\begin{proposition}[Deninger]\label{Deninger}
For an operator $\Phi$ in a finite von Neumann algebra with a trace $\tau(1)=1$ the following formula holds:
\begin{equation}
\Delta(z-\Phi)=|z|,
\end{equation}
if the following two conditions are satisfied:
\begin{equation}\label{spectral radius}
r(\Phi)<|z|
\end{equation}
and:
\begin{equation}\label{zero trace}
\tau(\Phi^n)=0,\quad \forall\mbox{ }n\geq 1.
\end{equation}
where $r(\Phi)$ stands for the the spectral radius of the operator $\Phi$.
\end{proposition}
\begin{remark}
The case $n=2$ in (\ref{operator}) can be dealt with by following the methods in \cite{Den1} provided that $g_1$ and $g_2$ are full Borel isomorphisms, that is $g_i:X\rightarrow X$. Assuming further that $g_1^{-1}g_2$ or $g_2^{-1}g_1$ is ergodic then one can embed the calculation of the determinant in the hyperfinite $II_1$-factor generated by the $\mathbf{Z}$-action of $g_1^{-1}g_2$, or $g_2^{-1}g_1$ on the probability space $X$ (notice that the ergodicity of an equivalence relation does not guarantee that of a subrelation). Thus our results will be relevant for the case $n\geqslant 3$. Moreover we will allow all of the $g_i$ but one to be Borel partial isomorphisms. However this level of generality has a slight drawback in the fact that we will impose some requirements on the domains of either the  $f_i$ or $g_i$. 
\end{remark}
\section{Main Results}
We will first need the following simple but important observation:
\begin{lemma}
Let $g:A\rightarrow B$ a Borel partial isomorphism and denote by $fg^{-1}$ the function $(fg^{-1})(x)=f(g^{-1}x)$. Then we have:

\begin{equation} \label{commutation}
L_{g}M_f=M_{fg^{-1}}L_{g}
\end{equation}

\begin{equation} \label{pow}
(M_fL_g)^n=M_{f\cdot fg^{-1}\cdot ...\cdot fg^{-(n-1)}}L_{g^n}
\end{equation}
\end{lemma}

\begin{proof}
This follows from a direct computation:
\[
(L_{g}M_f)(\Psi)(x,y)=L_{g}(f(\cdot)\Psi(\cdot,\cdot))(x,y)=\chi_{B}(x)f(g^{-1}x)\Psi(g^{-1}x,y)
\]
On the other hand:
\begin{align*}
(M_{fg^{-1}}L_{g})(\Psi)(x,y)&=M_{fg^{-1}}(\chi_{B}(\cdot)\Psi(g^{-1}\cdot,\cdot))(x,y)\\
                                 &=f(g^{-1}x)\chi_{B}(x)\Psi(g^{-1}x,y)
\end{align*}
Next, we can write
$(M_fL_g)^n=M_f (L_gM_f)^{n-1}L_g$.
Now (\ref{pow}) follows by applying (\ref{commutation}) repeatedly.
\end{proof}

\begin{proposition}\label{trace}
Let  $\Lambda$ be a treeing and assume that the equivalence relation $\mathcal{R}_{\Lambda}$  is ergodic. If $N$ is a positive integer and $g_i\in\Lambda$, $f_i\in L^{\infty}(X)$  $\forall\mbox{} i\in\{1,...,N\}$, then for $\Phi=\sum_{i=1}^N M_{f_i}L_{g_i}$ we have

\[
\tau(\Phi^n)=0,\quad\forall n\geqslant 1
\]
\end{proposition}

\begin{proof}
Notice that $\Phi^n$ will be a sum of terms having the format $M_{f_{i_1}}\circ L_{g_{i_1}}\circ \cdots \circ M_{f_{i_n}}\circ L_{g_{i_n}}$ where $i_1,\ldots ,i_n$ belong to the index set $\{1,2,\ldots,N\}$. Applying (\ref{commutation}) succesively we obtain that these terms are equal to an operator having the form
$M_hL_{g_1 g_2\ldots g_k}$, where $h=f_{i_1}\cdot f_{i_2}(g_{i_1}^{-1})\cdot f_{i_3}(g_{i_2}^{-1}g_{i_1}^{-1})\cdots f_{i_k}(g_{i_{k-1}}^{-1}\ldots g_{i_1}^{-1})$.
This shows it suffices to calculate the trace of the operators $M_hL_{\omega}$, where  $\omega=g_{i_1}g_{i_2}\ldots g_{i_k}$.
We have:
\begin{align*}
\tau(M_hL_{\omega})&=<M_hL_{\omega}\delta_0,\delta_0>=\int_X \sum_{z\sim x} \left( M_h(L_\omega \delta_0)(x,z)\overline{\delta_0(x,z)} \right) d\mu(x)\\
&=\int_X M_h(L_{\omega}\delta_0)(x,x)d\mu(x) = \int_{\{x \mid \omega(x)=x\} } h(x) d\mu(x)=0\\
\end{align*}
since the equivalence relation is treeable.
\end{proof}

\begin{teo}\label{r1} Let  $\Lambda$ be a treeing such that the equivalence relation $\mathcal{R}_{\Lambda}$  is ergodic. For $n\geq 1$ and $i\in\{1,\dots, n\}$  let $g_i\in\Lambda$, $f_i\in L^{\infty}(X)$, and $T=\sum_{i=1}^n M_{f_i}L_{g_i}$ in $\mathcal{M}(\mathcal{R}_{\Lambda})$.  Assume that there is an index $i_0$ such that
\begin{equation}\label{ine}
\sum_{i \neq i_0} || f_i/f_{i_0}||_{\infty}<1, 
\end{equation}
that $f_{i_0}$ is non-vanishing on sets of positive measure and that
\begin{equation}\label{nasty}
g_{i_0}: A_{i_0}=X \to B_{i_0}=X.
\end{equation}

Then:
\begin{equation}\label{detr1}
\log \Delta(T) = \int_X \log |f_{i_0}| d\mu(x)
\end{equation}
\end{teo}

\begin{proof}
Using (\ref{commutation}) we can rewrite the operator $T$ as:

\begin{equation}\label{factor}
T=M_{f_{i_0}}L_{g_{i_0}}\left[I+\sum_{i=1}^n M_{f_i/f_{i_0}g_{i_0}} L_{g_{i_0}^{-1}g_i} \right],
\end{equation}
where condition (\ref{nasty}) was used to ensure that $L_{g_{i_0}}L_{{g^{-1}_{i_0}}}=I$.\\
Notice that the operator $\Phi=\sum_{i \neq i_0}^n M_{f_i/f_{i_0}g_{i_0}} L_{g_{i_0}^{-1}g_i}$  satisfies (\ref{zero trace}) by Proposition \ref{trace}. Notice also that the spectral radius of $\Phi$ satisfies:
\begin{align*}
r(\Phi) &\leqslant \sum_{i \neq i_0} ||M_{f_i/f_{i_0}g_{i_0}}||\cdot ||L_{g_{i_0}}||\cdot ||L_{g_i}||\\
        &\leqslant  \sum_{i \neq i_0}||M_{f_i/f_{i_0}g_{i_0}}||\\
        & \leqslant \sum_{i \neq i_0} || f_i/f_{i_0}||_{\infty}<1
\end{align*}
where we used \ref{normLg} and \ref{normMf}. So \ref{spectral radius} is satisfied and we can apply \ref{Deninger} with $z=1$ to conclude that
\[
\Delta(I+\Phi) =1
\]
Taking the determinant on both sides of (\ref{factor}) we obtain (\ref{detr1}).
\end{proof}
\begin{remark}
In \cite{Den1} the equivalence relation $\mathcal{R}$ comes from the action of a single ergodic automorphism $\gamma:X\rightarrow X$ and among other things a formula for $\Delta(T)$ is found when $n=2$. This is done without the requirement (\ref{ine}). It would be interesting to extend the calculation of $\Delta(T)$ beyond the restriction (\ref{ine}).
\end{remark}
\noindent For our next result we will consider Borel partial isomorphisms \\
$g_i:A_i\rightarrow B_i$ that do not generate unitaries in $\mathcal{M(R)}$. In particular we will not be able to proceed as in Theorem \ref{r1}. However under some restrictions on the domains of the partial isomorphisms a calculation is still possible and even without the treeability assumption.

\begin{teo}\label{r2}
Let $f_i\in L^{\infty}(X)$ and $g_i:A_i\rightarrow B_i$, $i=1,..,n$ be Borel partial isomorphisms in the standard probability space  $(X,\mathcal{B}, \mu)$ such that the $g_i$ are among the generators of an (SP1) equivalence relation.
If the following conditions are satisfied
\begin{align}\label{s1}
\mu(A_1\cup A_2\ldots\cup A_n)=1
\end{align}
\begin{align}\label{s2}
\mu(B_i\cap B_j)=0\mbox{ if } i\neq j
\end{align}
then the Fuglede-Kadison determinant of the operator $T=\sum_{i=1}^n M_{f_i}L_{g_i}$ is given by
\begin{equation}\label{detr2}
\log \Delta(T) = \sum_{i=1}^{n}\int_{B_i} \log |f_{i}| d\mu(x)
\end{equation}

\end{teo}

\begin{proof}
First we will prove
\begin{equation}\label{s3}
\mu(A_i\cap A_j)=0\mbox{  if  }i\neq j
\end{equation}
Let $A:=\cup_{i=2}^nA_i$. We have:
\begin{align*}
1&\geqslant\mu(\cup_{i=1}^nB_i)=\sum_{i=1}^n\mu(B_i)=\sum_{i=1}^n\mu(A_i)\\
&\geqslant\mu(A_1)+\mu(A)=\mu(A_1\cup A)+\mu(A_1\cap A)\\
&=1+\mu(A_1\cap A)
\end{align*}
Hence $\mu(A_1\cap A)=0$ and (\ref{s3}) follows for $i=1$. Analogously we obtain it for all $i$. \\
Notice that the hypothesis $\mu(B_i\cap B_j)=0$ insures  $L_{g_i}^*L_{g_j}=0$ for $i\neq j$. Indeed, for $\psi\in L^2(\mathcal{R})$ :
$L_{g_i}^*L_{g_j}\psi(\cdot,\cdot)=\chi_{D_{g_j^{-1}g_i}}\psi(g_j^{-1}g_i\cdot,\cdot)=0$ because the  maximum domain of the word $g_j^{-1}g_i$ is  $D_{g_j^{-1}g_i}=g_i^{-1}(B_i\cap B_j)\cap A_i$. \\
Using this last observation and (\ref{commutation}) we show that $T^*T$ is a multiplication operator:\\
\begin{align*}
T^*T &=\sum_{i,j}L_{g_i}^*M_{\overline{f_i}f_j}L_{g_j}=\sum_{i,j}M_{\overline{f_i}f_jg_i}L_{g_i}^*L_{g_j}\\
&=\sum_iM_{|f_i|^2g_i}M_{\chi_{A_i}}=\sum_iM_{|f_i|^2g_i\cdot\chi_{A_i}}\\
&=M_{\sum_i |f_i|^2g_i\cdot\chi_{A_i}}
\end{align*}
Let's denote $h:=\sum_i |f_i|^2g_i\cdot\chi_{A_i}$. From the spectral theorem
\begin{align*}
\tau(\log{T^*T})=\int\log\lambda d\tau(E_{\lambda})
\end{align*}
where $(E_{\lambda})$ is the spectral resolution of the multiplication operator $M_h$. The spectral resolution has an explicit form in this case (see for example \cite{Co}): for any Borel set $A\subset\sigma(M_h)$, $E(A)=M_{\chi_{ h^{-1}A}}$ and $\tau(M_{\chi_{ h^{-1}A}})=\mu(h^{-1}A)$. Similar with the classic multiplication operators on $L^2(X)$ and the fact that on $\mathcal{M(R)}$ the trace is a vector-trace
we have
\begin{align*}
\int\log\lambda d\tau(E_{\lambda})=\int_X\log h(x)d\mu(x)
\end{align*}
Combining this with the definition of the determinant
\begin{align*}
2\log\Delta (T) = \tau(\log{T^*T})=\int_X\log h(x)d\mu(x)
\end{align*}
Because of (\ref{s1}) and (\ref{s3}) we can continue with
\begin{align*}
2\log\Delta (T) &= \sum_i\int_{A_i}\log h(x)d\mu(x)=\sum_i\int_{A_i}\log |f_ig_i(x)|^2d\mu(x)\\&=2\sum_i\int_{B_i}\log |f_i(x)|d\mu(x)
\end{align*}
Now (\ref{detr2}) follows.
\end{proof}
\begin{remark} The theorem above is true  if we switch (\ref{s1}) with (\ref{s3}). Similar arguments can be used to show that (\ref{s3}) and (\ref{s1}) imply (\ref{s2}).
\end{remark}

\noindent
\textbf{Acknowledgement}. The authors thank the referee for carefully reading this paper and making several helpful suggestions.


\begin{thebibliography}{99}
\bibitem[Bow]{Bow}
L. Bowen, \emph{ Entropy for expansive algebraic actions of residually finite groups}, Ergodic Theory Dynam. Systems, 31 (2011), p.703-718
\bibitem[Co]{Co}
J.B. Conway,  \emph{ A course in functional analysis, 2nd edition}, Springer Verlag, 1990
\bibitem[Den1]{Den1}
C. Deninger, \emph{ Determinants on von Neumann algebras, Mahler measures and Ljapunov exponents}, J. Reine Angew. Math. 651,(2011), p.165-185
\bibitem[Den2]{Den2}C. Deninger, \emph{ Fuglede-Kadison determinants and entropy for actions of discrete
amenable groups}, J. Amer. Math. Soc. 19 (2006), p.737-758
\bibitem[FMI]{FMI}
J. Feldman, C. Moore, \emph{ Ergodic Equivalence Relations, Cohomology, and von Neumann
Algebras.I}, Transactions of the AMS, dec. 1977, vol.234, issue 2, p.289-324
\bibitem[FMII]{FMII}
J. Feldman, C. Moore, \emph{ Ergodic Equivalence Relations, Cohomology, and von Neumann
Algebras.II}, Transactions of the AMS, dec. 1977, vol.234, issue 2, p.325-359
\bibitem[FK]{FK}
B. Fuglede, R.V. Kadison,\emph{ Determinant theory in finite factors}, Ann. Math.55 (1952), p.520-530
\bibitem[Gab]{Gab}
D. Gaboriau,\emph{Co\^{u}t des relations d'equivalence et des groupes}, Invent. Math.139, (2000), p.41-98
\bibitem[HSc]{HSc}
 U. Haagerup, H. Schultz, \emph{ Invariant subspaces for operators in a general $II_1$-
factor}, Publ. Math. Inst. Hautes Etudes Sci. 109 (2009), p.19-111
\bibitem[KLi]{KLi}
D. Kerr, H. Li, \emph{ Entropy and the variational principle for actions of sofic groups}, Invent.Math.186, (2011), p.501-558
\end{thebibliography}
\end{document}